\documentclass[a4paper,11pt]{amsart}

\usepackage{amscd,amsthm,amsfonts,amssymb,amsmath}
\usepackage[colorlinks=true]{hyperref}
\usepackage{fullpage}
\usepackage[all]{xy}

\usepackage{setspace}

    \newcommand{\BC}{{\mathbb {C}}} 
     \newcommand{\BF}{{\mathbb {F}}}

     \newcommand{\BP}{{\mathbb {P}}}
    \newcommand{\BQ}{{\mathbb {Q}}} \newcommand{\BR}{{\mathbb {R}}}

     \newcommand{\BZ}{{\mathbb {Z}}}

     \newcommand{\CH}{{\mathcal {H}}}

    \newcommand{\CS}{{\mathcal {S}}}

     \newcommand{\fL}{{\mathfrak{L}}}

    \newcommand{\loc}{{\mathrm{loc}}}

    \newcommand{\ord}{{\mathrm{ord}}}

    \renewcommand{\mod}{\ \mathrm{mod}\ }

    \font\cyr=wncyr10

    \newcommand{\Sha}{\hbox{\cyr X}}

    \theoremstyle{plain}
    \newtheorem{thm}{Theorem}[section] \newtheorem{cor}[thm]{Corollary}
    \newtheorem{lem}[thm]{Lemma}  \newtheorem{prop}[thm]{Proposition}
    \newtheorem {conj}[thm]{Conjecture}

\theoremstyle{remark} 
\theoremstyle{remark} 
\theoremstyle{remark} 

    \numberwithin{equation}{section}

\begin{document}


\title{Non-vanishing theorems for quadratic twists of elliptic curves}

\author{Shuai Zhai}



\begin{abstract}
In this paper, we use rather classical results on modular symbols to prove that, for certain families of elliptic curves defined over $\BQ$, there always exists a large class of explicit quadratic twists whose complex $L$-series does not vanish at $s=1$. We also prove the $2$-part of the conjecture of Birch and Swinnerton-Dyer for many of these quadratic twists.
\end{abstract}

\maketitle

\tableofcontents

\section{Introduction}

Let $E$ be an elliptic curve defined over $\BQ$, and let $L(E,s)$ be the complex $L$-series of $E$. For each square-free non-zero integer $d \neq 1$, we write $E^{(d)}$ for the twist of $E$ by the quadratic extension $\BQ(\sqrt{d})/\BQ$, and $L(E^{(d)},s)$ for its complex $L$-series. Let $C_E$, or simply $C$ when there is no danger of confusion, denote the conductor of $E$. As usual, $\Gamma_0(C)$ will denote the subgroup of $SL_2(\BZ)$ consisting of all matrices with the bottom left hand corner entry divisible by $C$, and we write $X_0(C)$ for the corresponding modular curve. It is known that, by the theorem of Wiles \cite{Wiles}, Taylor--Wiles \cite{Taylor} and Breuil--Conrad--Diamond--Taylor \cite{Breuil}, all elliptic curves $E/\BQ$ have a modular parametrization, i.e. there is a non-constant rational map $\phi$ from the modular curve $X_0(C)$ to $E$, which maps the cusp at infinity to the zero element of $E$. We say that $E$ is an \emph{optimal} elliptic curve if the map $\phi$ does not factor through any other elliptic curve. There is an optimal curve in every isogeny class of elliptic curves defined over $\BQ$, and, throughout the present paper, we shall always assume that $E$ is indeed optimal. The pull-back by $\phi$ of a N\'{e}ron differential on a global minimal Weierstrass equation for $E$ is then given by a rational multiple, whose absolute value we denote by $\nu_E$, of the differential associated to the normalized new (or primitive) cuspidal eigenform $f = f_E$ of weight $2$ and level $C = C_E$ associated with $E$. It is known \cite{Edixhoven} that $\nu_E$ is always an integer, and Manin conjectured that in fact $\nu_E = 1$, and Cremona \cite{Cremona} has verified that this is true for all $E$ with conductor $C_E \leq 60,000$. Throughout the present paper we shall, for simplicity, always make the standing assumption:-

\medskip

\textbf{Assumption}. \emph{The Manin constant $\nu_E$ is odd}.

\medskip

We remark that Abbes--Ullmo \cite{Abbes} have proven this assumption whenever $C_E$ is an odd integer. In the present paper, we shall be interested in studying the Birch--Swinnerton-Dyer conjecture for quadratic twists $E^{(M)}$ of $E$, where $M$ will always be assumed to be the a square-free, positive or negative integer with $M \equiv 1 \mod 4$. Let $\Omega_{E^{(M)}}^+$ denote the least positive real period of a N\'{e}ron differential on a global minimal Weierstrass equation for $E^{(M)}$, and define
$$
L^{(alg)}(E^{(M)}, 1) = L(E^{(M)}, 1)/\Omega_{E^{(M)}}^+.
$$
It is well known that $L^{(alg)}(E^{(M)}, 1)$ is a rational number. We write $ord_2$ for the order valuation of $\BQ$ at the prime $2$, with the normalization $ord_2(2) = 1$. Also we define $ord_2(0) = \infty$. Let $f(x)$ be the $2$-division polynomial of $E$. When $f(x)$ is irreducible over $\BQ$, we define $F$ to be the field obtained by adjoining to $\BQ$ one fixed root of $f(x)$. Let $q$ be any prime of good reduction for $E$, and let $a_q$ be the trace of Frobenius at $q$ on $E$ and denote $N_q:=1+q-a_q$. For each integer $m > 1$, let $E[m]$ denote the group of $m$-division points on $E$. Also, we define a rational prime $q$ to be inert in the field $F$ if it is unramified and there is a unique prime of $F$ above $q$. By applying some results by Manin \cite{Manin} and Cremona \cite{Cremona} on modular symbols, we prove the following general results. Of course, we make use below of the fundamental theorem of Kolyvagin \cite{Kolyvagin} which asserts that, if the complex $L$-series of an elliptic curve defined over $\BQ$ does not vanish at the point $s=1$, then both the Mordell--Weil group and the Tate--Shafarevich group of the curve are finite. Throughout this paper, we always assume that the Manin constant is always odd. By \cite{Abbes} this is no assumption at all if the conductor of the elliptic curve is odd, and the conjecture that the Manin constant is always $1$ has been verified numerically by Cremona for all curves of conductor less than $60000$.

\medskip

We first give results for curves $E$ with $E[2](\BQ) = 0$.

\begin{thm}\label{thm1}
Let $E$ be an optimal elliptic curve over $\BQ$, with odd Manin constant. Assume that $E$ has negative discriminant, and satisfies $E[2](\BQ)=0$ and $ord_2(L^{(alg)}(E,1))=0$. Let $M$ be any integer of the form $M=\epsilon q_1q_2 \cdots q_{r}$, satisfying $(M,C)=1$, where $C$ is the conductor of $E$, $r \geq 1$, $q_1, \ldots, q_{r}$ are arbitrary distinct odd primes which are inert in the field $F$, and the sign $\epsilon = \pm 1$ is chosen so that $M \equiv 1 \mod 4$. Then $L(E^{(M)},1) \neq 0$, and we have
$$
ord_2(L^{(alg)}(E^{(M)},1))=0.
$$
Hence, $E^{(M)}(\BQ)$ and $\Sha(E^{(M)}(\BQ))$ are finite.
\end{thm}

Let $Sel_2(E)$ denote the $2$-Selmer group of $E$ over $\BQ$. In view of our assumption that $E[2](\BQ)=0$, the $2$-part of the Birch and Swinnerton-Dyer conjecture for $E$ would show that our hypothesis that $ord_2(L^{(alg)}(E,1))=0$ implies that $Sel_2(E)=0$, but it is still not known how to prove this at present. However, if we assume that the $2$-part of the Birch and Swinnerton-Dyer conjecture holds for $E$, and in addition to the hypotheses of Theorem \ref{thm1}, we will have the following result under one further condition.

\begin{thm}\label{thm1-bsd}
Assume the hypotheses of Theorem \ref{thm1}. We also suppose that the bad primes of $E$ all split in $\BQ(\sqrt{M})$, and that the $2$-part of the Birch and Swinnerton-Dyer conjecture holds for $E$. Then the $2$-part of the Birch and Swinnerton-Dyer conjecture holds for all the twists $E^{(M)}$.
\end{thm}

We remark here under the hypotheses of Theorem \ref{thm1}, if we also suppose that (i) the bad primes of $E$ are all primes of multiplicative reduction and the minimal discriminant for $E$ is divisible by each of these primes to an odd power, and (ii) $E$ has good reduction at $2$ and the reduction of $E$ modulo $2$ has $j$-invariant $0$, then the work of Boxer and Diao \cite[Theorem 1.2]{Boxer}, combined with Theorem \ref{thm1} does indeed prove the $2$-part of the Birch and Swinnerton-Dyer conjecture for all the twists $E^{(M)}$, with $M$ as in Theorem \ref{thm1}. Finally, we remark that it is an interesting exercise to verify directly that, for all the twists $E^{(M)}$, with $M$ as in Theorem \ref{thm1}, the global root number of $L(E^{(M)},s)$ is $+1$.

Of course, the Chebotarev theorem shows that there is a positive density of primes which are inert in $F$. Here are some examples of curves to which Theorem \ref{thm1} applies, such as $X_0(11)$, which we view as an elliptic curve by taking $[\infty]$ to be the origin of the group law, and which has a minimal Weierstrass equation given by
$$
E: y^2 + y = x^3 - x^2 - 10x - 20.
$$
Moreover, $E(\BQ) \cong \BZ/5\BZ$, $L^{(alg)}(E,1)=\frac{1}{5}$, and it has discriminant $-11^5$. A simple polynomial defining the $2$-division field is $f(x)=x^3-x^2+x+1$, which has discriminant $-44$. Here are a list of odd primes which are inert in the field $F$:
$$
3, 5, 23, 31, 37, 59, 67, 71, 89, 97, 113, 137, 157, 179, 181, 191, \ldots.
$$
The $2$-part of Birch and Swinnerton-Dyer conjecture is valid for all these twists. Further examples of elliptic curves $E$ to which Theorem \ref{thm1} applies are as follows (we use Cremona's label for each curve). First we can take $E=X_0(19)$, which has conductor $19$ and equation
$$
19A1: y^2 + y = x^3 + x^2 - 9x - 15.
$$
Also we can take the curves
$$
26A1: y^2 + xy + y = x^3 - 5x - 8, \ \text{and} \ 26B1: y^2 + xy + y = x^3 - x^2 - 3x + 3,
$$
which have conductor $26$, and the curves
$$
121A1: y^2 + xy + y = x^3 + x^2 - 30x - 76, \ \text{and} \ 121C1: y^2 + xy = x^3 + x^2 - 2x - 7,
$$
which have conductor $121$.

\smallskip

When $E$ has positive discriminant, two entirely parallel results hold, provided we only consider twists by $\BQ(\sqrt{M})/\BQ$ with $M > 0$, and $M \equiv 1 \mod 4$. The hypothesis that $M$ should now be positive is needed to ensure that the global root number of $L(E^{(M)},s)$ is $+1$ for all the $M$ in the next theorem.

\begin{thm}\label{thm1-1}
Let $E$ be an optimal elliptic curve over $\BQ$, with odd Manin constant. Assume that $E$ has positive discriminant, and satisfies $E[2](\BQ)=0$ and $ord_2(L^{(alg)}(E,1))=1$. Let $M$ be any positive integer of the form $M=q_1q_2 \cdots q_{r}$, satisfying $(M,C)=1$, where $C$ is the conductor of $E$, $r \geq 1$, $q_1, \ldots, q_{r}$ are arbitrary distinct odd primes which are inert in the field $F$, and $M \equiv 1 \mod 4$. Then $L(E^{(M)},1) \neq 0$, and we have
$$
ord_2(L^{(alg)}(E^{(M)},1))=1.
$$
Hence, $E^{(M)}(\BQ)$ and $\Sha(E^{(M)}(\BQ))$ are finite.
\end{thm}

\begin{thm}\label{thm1-1-bsd}
Assume the hypotheses of Theorem \ref{thm1-1}. We also suppose that the bad primes of $E$ all split in $\BQ(\sqrt{M})$, and that the $2$-part of the Birch and Swinnerton-Dyer conjecture holds for $E$. Then the $2$-part of the Birch and Swinnerton-Dyer conjecture holds for all the twists $E^{(M)}$.
\end{thm}

Here are some examples of curves to which Theorem \ref{thm1-1} applies. We can take $E=37B1$, which has a minimal Weierstrass equation given by
$$
E: y^2 + y = x^3 + x^2 - 23x - 50.
$$
Moreover, $E(\BQ) \cong \BZ/3\BZ$, $L^{(alg)}(E,1)=\frac{2}{3}$, and it has discriminant $37^3$. A simple polynomial defining the $2$-division field is $f(x)=x^3+x^2-3x-1$, which has discriminant $148$. Here are a list of odd primes which are inert in the field $F$:
$$
3, 7, 11, 41, 47, 53, 71, 73, 83, 101, 127, 149, 157, 173, 181, 197, \ldots.
$$
Further examples of elliptic curves $E$ to which Theorem \ref{thm1-1} applies are as follows. First we can take $141E1$, which has conductor $141$ and equation
$$
141E1: y^2 + y = x^3 + x^2 - 26x - 61,
$$
and also we can take the curves
$$
142D1: y^2 + xy = x^3 - 8x + 8, \ \text{and} \ 142E1: y^2 + xy = x^3 - x^2 - 2626x + 52244,
$$
which have conductor $142$.

\medskip

For curves $E$ with $E[2](\BQ) \neq 0$, we have only been able to establish the following much weaker results in which we only consider twists by $\BQ(\sqrt{M})/\BQ$ with $M \equiv 1 \mod 4$ and divisible by only one prime. When we take the value of $M$, we have to make sure that the global root number of the complex $L$-series of $E^{(M)}$ is $+1$.

\begin{thm}\label{thm2}
Let $E$ be an optimal elliptic curve over $\BQ$, with odd Manin constant. Assume that $E$ has negative discriminant, and satisfies $E[2](\BQ) \neq 0$ and $L(E,1) \neq 0$. Let $M$ be any integer of the form $M=\epsilon q$, where $q$ is an arbitrary odd prime with $(q,C)=1$, where $C$ is the conductor of $E$, and the sign $\epsilon = \pm 1$ is chosen so that $M \equiv 1 \mod 4$. If $ord_2(N_q)=-ord_2 (L^{(alg)}(E,1)) \neq 0$, then $L(E^{(M)},1) \neq 0$, and we have
$$
ord_2(L^{(alg)}(E^{(M)},1))=0.
$$
Hence, $E^{(M)}(\BQ)$ and $\Sha(E^{(M)}(\BQ))$ are finite.
\end{thm}

Note that, in the above theorem, we are assuming, in particular, that $ord_2(L^{(alg)}(E,1)) < 0$. It is not known at present how to deduce from this assumption that $E[2](\BQ)$ is non-zero, although of course this would follow from the conjecture of Birch and Swinnerton-Dyer for $E$. Here are some examples of curves to which Theorem \ref{thm2} applies, such as the Neumann--Setzer elliptic curves, which have conductor $p$, where $p$ is a prime of the form $u^2+64$ for some integer $u \equiv 1 \mod 4$. We denote such a curve by $A$, and recall that it has a global minimal equation given by 
$$
A: y^2 + xy = x^3 + \frac{u-1}{4}x^2 + 4x + u.
$$
We shall consider all these curves in detail and prove the following theorem in Section 5.
\begin{thm}\label{thmA}
Let $q$ be any prime congruent to $3$ modulo $4$ and inert in $\BQ(\sqrt{p})$. When $u \equiv 5 \mod 8$, then $L(A^{(-q)},1) \neq 0$, and we have
$$
ord_2(L^{(alg)}(A^{(-q)},1))=0.
$$
Hence, $A^{(-q)}(\BQ)$ is finite, the Tate--Shafarevich group $\Sha(A^{(-q)}(\BQ))$ is finite of odd cardinality. Moreover, the $2$-part of Birch and Swinnerton-Dyer conjecture is valid for $A^{(-q)}$.
\end{thm}

Here we take $X_0(17)$ as another example, which has a minimal Weierstrass equation given by
$$
E: y^2 + xy + y = x^3 - x^2 - x - 14.
$$
Moreover, $E(\BQ) \cong \BZ/4\BZ$, $L^{(alg)}(E,1)=\frac{1}{4}$, and it has discriminant $-17^4$. In particular, our theorem applies to all primes $q$ with $q \equiv 3 \mod 4$ and which are inert in $\BQ(\sqrt{17})$, whence $N_q \equiv 4 \mod 8$. Here are a list of odd primes satisfying the above conditions:
$$
3, 7, 11, 23, 31, 71, 79, 107, 131, 139, 163, 167, 199, \ldots.
$$
The Chebotarev theorem shows that there is a positive density of primes which are both inert in $\BQ(\sqrt{i})$ and $\BQ(\sqrt{17})$. For the twists $E^{(-q)}$ for such primes $q$, it is easy to show by a classical $2$-descent that $E^{(-q)}(\BQ)$ is finite and that $\Sha (E^{(-q)}(\BQ))[2] = 0$. Thus the $2$-part of the Birch--Swinnerton-Dyer conjecture is valid for $E^{(-q)}$. Further examples of elliptic curves $E$ to which Theorem \ref{thm2} applies are as follows. First we can take $E=X_0(14)$, which has conductor $14$ and equation
$$
14A1: y^2 + xy + y = x^3 + 4x - 6,
$$
also we can take the curve $X_0(49)$, which has conductor $49$ and equation
$$
49A1: y^2 + xy = x^3 - x^2 - 2x - 1,
$$
and which has been fully investigated by Coates, Li, Tian, and Zhai by Zhao's method and Waldspurger's formula (see \cite{Coates1}).

\smallskip

Similarly, when $E$ has positive discriminant, an entirely parallel result holds, provided we only consider twists by $\BQ(\sqrt{q})/\BQ$ with some prime $q \equiv 1 \mod 4$.

\begin{thm}\label{thm2-1}
Let $E$ be an optimal elliptic curve over $\BQ$, with odd Manin constant. Assume that $E$ has positive discriminant, and satisfies $E[2](\BQ) \neq 0$ and $L(E,1) \neq 0$. Let $q$ be any odd prime with $q \equiv 1 \mod 4$, and $(q,C)=1$, where $C$ is the conductor of $E$. If $ord_2(N_q) = 1-ord_2 (L^{(alg)}(E,1)) \neq 0$, then $L(E^{(q)},1) \neq 0$, and we have
$$
ord_2(L^{(alg)}(E^{(q)},1))=1.
$$
Hence, $E^{(q)}(\BQ)$ and $\Sha(E^{(q)}(\BQ))$ are finite.
\end{thm}

Here are some examples of curves to which Theorem \ref{thm2-1} applies, such as $X_0(21)$, which has a minimal Weierstrass equation given by
$$
E: y^2 + xy = x^3 - 4x - 1.
$$
Moreover, $E(\BQ) \cong \BZ/2\BZ \times \BZ/4\BZ$, $L^{(alg)}(E,1)=\frac{1}{4}$, and it has discriminant $3^4\cdot7^2$. In particular, our theorem applies to all primes $q$ with $q \equiv 1 \mod 4$ and which are both inert in $\BQ(\sqrt{3})$ and $\BQ(\sqrt{7})$, whence $N_q \equiv 8 \mod 16$. Here are a list of odd primes satisfying the above conditions:
$$
5, 17, 41, 89, 101, 173, 269, 293, \ldots.
$$
The Chebotarev theorem shows that there is a positive density of primes congruent to $1$ modulo $4$ which are both inert in $\BQ(\sqrt{3})$ and $\BQ(\sqrt{7})$. For the twists $E^{(q)}$ for such primes $q$, it is easy to show that $E^{(q)}(\BQ)$ is finite and that $\Sha (E^{(q)}(\BQ))[2] = 0$. Thus the $2$-part of the Birch--Swinnerton-Dyer conjecture is valid for $E^{(q)}$. Further examples of elliptic curves $E$ to which Theorem \ref{thm2-1} applies are as follows. First we can take $E=X_0(33)$, which has conductor $33$ and equation
$$
33A1: y^2 + xy = x^3 + x^2 - 11x,
$$
also we can take $E=X_0(34)$, which has conductor $34$ and equation
$$
34A1: y^2 + xy = x^3 - 3x + 1.
$$

It is not difficult to see that Theorems \ref{thm2} and \ref{thm2-1} are also entirely consistent with the $2$-part of the conjecture of Birch and Swinnerton-Dyer. Moreover, it is easy to carry out a straightforward classical $2$-descent on these curves because of our hypothesis that $E[2](\BQ) \neq 0$. Then after considering the behaviour of Tamagawa factors under twisting (see the lemma at the end of Section 4), one can verify the $2$-part of Birch and Swinnerton-Dyer conjecture for all these curves.

\medskip

For curves $E$ with $ord_2(L^{(alg)}(E,1)) \neq 0$ and negative discriminant, we could obtain the following lower bound for some twists of $E$.

\begin{thm}\label{thm3}
Let $E$ be an optimal elliptic curve over $\BQ$, with odd Manin constant. Assume that $E$ has negative discriminant, and satisfies $E[2](\BQ) \neq 0$ and $L(E,1) \neq 0$. Let $M$ be any integer of the form $M=\epsilon q_1q_2 \cdots q_{r}$ and satisfying $(M,C)=1$, where $C$ is the conductor of $E$, $r \geq 1$, $q_1, \ldots, q_{r}$ are arbitrary distinct odd primes, and the sign $\epsilon = \pm 1$ is chosen so that $M \equiv 1 \mod 4$. If $ord_2(N_{q_i}) > -ord_2 (L^{(alg)}(E,1))$ holds for at least one prime factor $q_i$ ($1 \leq i \leq r$) of $M$, then we have
$$
ord_2(L^{(alg)}(E^{(M)},1)) \geq 1.
$$
\end{thm}

We remark that Theorem \ref{thm3} can apply to all optimal elliptic curves $E$ with negative discriminant, and satisfying $L(E,1) \neq 0$. When $E$ has positive discriminant, we have the following trivial lower bound result.

\begin{thm}\label{thm3-1}
Let $E$ be an optimal elliptic curve over $\BQ$, with odd Manin constant. Assume that $E$ has positive discriminant, and satisfies $E[2](\BQ) \neq 0$ and $L(E,1) \neq 0$. Let $M \neq 1$ be any integer with $M \equiv 1 \mod 4$ and $(M,C)=1$, where $C$ is the conductor of $E$. Then we have
$$
ord_2(L^{(alg)}(E^{(M)},1)) \geq 1.
$$
\end{thm}

We remark that the integer $ord_2 (L^{(alg)}(E,1))+ord_2(N_q)$ could not be negative by an easy observation of Manin's modular symbol formula, which will be talked about in the following section.

\medskip

In conclusion, I am extremely grateful to my supervisor John Coates, and to John Cremona for his very helpful remarks on the questions discussed in this paper. I also would like to thank the China Scholarship Council for supporting my studies in the Department of Pure Mathematics and Mathematical Statistics, University of Cambridge.

\bigskip

\section{Modular symbols}

Modular symbols were first introduced by Birch and Manin \cite{Manin} several decades ago and since then have been studied, refined, and reformulated by several authors. They provide an explicit description of classical modular forms by a finite set of algebraic integers, and thus are the main tool for computations of modular forms. We shall show in this paper that they are also very useful in studying the 2-part of the conjecture of Birch and Swinnerton-Dyer.

For each integer $C \geq 1$, let $S_2(\Gamma_0(C))$ be the space of cusp forms of weight $2$ for $\Gamma_0(C)$. In this section, we are focusing on the modular forms in the space $S_2(\Gamma_0(C))$, which is closely corresponding to elliptic curves and could be computed in terms of modular symbols. In what follows, $f$ will always denote a normalized primitive eigenform in $S_2(\Gamma_0(C))$, all of whose Fourier coefficients belong to $\BQ$. Thus $f$ will correspond to an isogeny class of elliptic curves defined over $\BQ$, and we will denote by $E$ the unique {\it optimal} elliptic curve in the $\BQ$-isogeny class of $E$. The complex $L$-series $L(E,s)$ will then coincide with the complex $L$-series attached to the modular form $f$. Moreover, there will be a non-constant rational map defined over $\BQ$
$$
\phi: X_0(C) \to E,
$$
which does not factor through any other elliptic curve in the isogeny class of $E$. Let $\omega$ denote the N\'{e}ron differential on a global minimal Weierstrass equation for $E$. Then, writing $\phi^*(\omega)$ for the pull back of $\omega$ by $\phi$, there exists $\nu_E \in \BQ^\times$ such that
\begin{equation}\label{mc}
\nu_Ef(\tau)d\tau = \phi^*(\omega).
\end{equation}
The rational number $\nu_E$ is called the {\it Manin constant}. It is well known to lie in $\BZ$, and it is conjectured to always be equal to 1. Moreover, it is known to be odd whenever the conductor $C$ of $E$ is odd. Let $\CH$ be the upper half plane, and put $\CH^*=\CH\cup\BP_1(\BQ)$. Let $g$ be any element of  $\Gamma_0(C)$. Let $\alpha, \beta$ be two points in $\CH^*$ such that $\beta=g\alpha$. Then any path from $\alpha$ to $\beta$ on $\CH^*$ is a closed path on $X_0(C)$ whose homology class only depends on $\alpha$ and $\beta$. Hence it determines an integral homology class in $H_1(X_0(C),\BZ)$, and we denote this homology class by the \emph{modular symbol} $\{\alpha,\beta\}$. Let $\alpha,\beta,\gamma \in \CH^*$ and $g,g_1,g_2 \in G$, we will have the following properties which could be obtained easily by the definition and one can find a proof in \cite[Chapter 2]{Cremona} and \cite{Manin}:

1) $\{\alpha,\alpha\}=0$;

2) $\{\alpha,\beta\}+\{\beta,\alpha\}=0$;

3) $\{\alpha,\beta\}+\{\beta,\gamma\}+\{\gamma,\alpha\}=0$;

4) $\{g\alpha,g\beta\}_G=\{\alpha,\beta\}_G$;

5) $\{\alpha,g\alpha\}_G=\{\beta,g\beta\}_G$;

6) $\{\alpha,g_1g_2\alpha\}_G=\{\alpha,g_1\alpha\}_G+\{\alpha,g_2\alpha\}_G$;

7) $\{\alpha,g\alpha\}_G \in H_1(X_0(C),\BZ)$.

\smallskip

We can then form the modular symbol
$$
\langle \{\alpha,\beta\}, f \rangle := \int_{\alpha}^{\beta} 2\pi i f(z) dz.
$$
The period lattice $\Lambda_f$ of the modular form $f$ is defined to be the set of these modular symbols for all such pairs  $\{\alpha,\beta\}$. It is a discrete subgroup of $\BC$ of rank $2$. If $\fL_E$ denotes the period lattice of the N\'{e}ron differential $\omega$ on $E$, it follows from \eqref{mc}
that
\begin{equation}\label{mc2}
\fL_E = \nu_E\Lambda_f.
\end{equation}
Define define $\Omega_E^+$ (resp. $i \Omega_E^-$) to be the least positive real (resp. purely imaginary) period of the N\'{e}ron differential of a global minimal equation for $E$, and $\Omega_f^+$ (resp. $i \Omega_f^-$) to be the least positive real (resp. purely imaginary) period of $f$. Thus, by \eqref{mc2}, we have
\begin{equation}\label{mc3}
\Omega_E^+= \nu_E\Omega_f^+, \ \ \ \Omega_E^-= \nu_E\Omega_f^-.
\end{equation}
In this section, we will carry out all of our computations with the period lattice $\Lambda_f$, but whenever we subsequently translate them into assertions about the conjecture of Birch and Swinnerton-Dyer for the elliptic curve $E$, we must switch to the period lattice $\fL_E$ by making use of \eqref{mc2}.

More generally, if $\alpha, \beta$ are any two elements of $\CH^*$, and $g$ is any element of $S_2(\Gamma_0(C))$, we put $\langle \{\alpha,\beta\}, g \rangle := \int_{\alpha}^{\beta} 2\pi i g(z) dz$.  This linear functional defines an element of $H_1(X_0(C), \BR)$, which we also denote by $\{\alpha, \beta\}.$ The following theorem was proven by Manin  and Drinfeld \cite{Manin}:
\begin{thm} (Manin--Drinfeld). For all pairs of cusps $\alpha,\beta \in \CH^*$, we have
$$
\{\alpha,\beta\} \in H_1(X_0(C),\BQ).
$$
\end{thm}

Let $m$ be a positive integer satisfying $(m,C)=1$. According to Birch, Manin \cite[Theorem 4.2]{Manin} and Cremona \cite[Chapter 3]{Cremona}, we have the following formulae:
\begin{equation}\label{ms1}
(\sum_{l \mid m}l-a_m) L(E,1)=-\sum_{\substack{l \mid m \\ k \mod l}} \langle \{0,\frac{k}{l}\}, f \rangle;
\end{equation}
here $l$ runs over all positive divisors of $m$; and
\begin{equation}\label{ms2}
L(E,\chi,1)=\frac{g(\bar{\chi})}{m} \sum_{k \mod m}{\chi}(k) \langle \{0,\frac{k}{m}\}, f \rangle;
\end{equation}
here $\chi$ is any primitive Dirichlet character modulo $m$, and $g(\bar{\chi})= \sum_{k \mod m} \bar{\chi}(k) e^{2\pi i \frac{k}{m}}$.

For each odd square-free positive integer $m$, we define $r(m)$ to be the number of prime factors of $m$. Also, in what follows, we shall always only consider the positive divisors of $m$. Also, in what follows, we always only consider positive divisors of $m$. We define
$$
S_m:=\sum_{k=1}^{m} \langle \{0,\frac{k}{m}\}, f \rangle,
$$
$$
S'_m:=\sum_{\substack{k=1 \\ (k,m)=1}}^{m} \langle \{0,\frac{k}{m}\}, f \rangle.
$$
We repeatedly use the following identity.

\smallskip

\begin{lem}\label{sum_Sl}
For each odd square-free positive integer $m>1$, we have
$$
\sum_{l \mid m} S_l = \sum_{d=1}^{r(m)} 2^{r(m)-d} \sum_{\substack{n \mid m \\ r(n)=d}} S'_n.
$$
\end{lem}
\begin{proof}
We obviously have
$$
\sum_{l \mid m} S_l = \sum_{l \mid m} \sum_{n \mid l} S'_n.
$$
Now fix an integer $d$ with $1 \leq d \leq r(m)$, and divisor $n$ of $m$ with $r(n)=d$. Then the number of divisors $l$ of $m$, which are divisible by $n$, is equal to
$$
2^{r(m)-d},
$$
whence the assertion of the lemma follows.
\end{proof}

\smallskip

If $q$ is any prime of good reduction for $E$, we let $a_q$ denote the trace of Frobenius at $q$, and define $N_q=q+1-a_q$. Thus $N_q$ is the number of points on the reduction of $E$ modulo $q$ with coordinates in the field with $q$ elements. Now suppose that $m=q_1q_2 \cdots q_{r(m)}$ is an odd square-free integer $m>1$ with $(m,C)=1$. The following identity is due to \eqref{ms1}
\begin{equation}\label{ms_qr}
\left((1+q_1)(1+q_2) \cdots (1+q_{r(m)})-a_{q_1}a_{q_2} \cdots a_{q_{r(m)}}\right) L(E,1) = -\sum_{l \mid m} S_l.
\end{equation}

\smallskip

\begin{lem}\label{lemmaS'_m}
Let $E$ be a $\Gamma_0(C)$-optimal elliptic curve over $\BQ$, with $L(E,1) \neq 0$ and $E[2](\BQ)=0$. Let $m$ be an odd square-free integer greater than $1$ with $(m,C)=1$. Assume that $N_q$ is odd for each prime factor $q$ of $m$. Then
$$
ord_2(S'_m/\Omega_f^+) = ord_2 (L(E,1)/{\Omega_f^+}).
$$
\end{lem}
\begin{proof}
We use induction on $r(m)$, the number of prime factors of $m$. Suppose first that $r(m)=1$, say $m=q_1$. Then by \eqref{ms_qr}, we have
$$
N_{q_1} L(E,1) = -S'_{q_1},
$$
and the assertion is then clear because $N_{q_1}$ is odd. Now suppose $r(m)>1$ and assume the lemma is true for all divisors $n>1$ of $m$ with $n \neq m$. Note also that $a_{q_1}, \ldots, a_{q_{r(m)}}$ are all odd, and so
$$
(1+q_1)(1+q_2) \cdots (1+q_{r(m)})-a_{q_1}a_{q_2} \cdots a_{q_{r(m)}}
$$
is odd. Hence it follows from \eqref{ms_qr} that
$$
ord_2 (L(E,1)/{\Omega_f^+}) = ord_2 \left(\sum_{l \mid m} S_l / \Omega_f^+ \right).
$$
But, by Lemma \ref{sum_Sl}, we have
$$
\sum_{l \mid m} S_l / \Omega_f^+ = S'_m / \Omega_f^+ + \sum_{d=1}^{r(m)-1} 2^{r(m)-d} \sum_{\substack{n \mid m \\ r(n)=d}} S'_n / \Omega_f^+.
$$
By our induction hypothesis, every term in the second sum on the right hand side of this equation has order strictly greater than $ord_2 (L(E,1)/{\Omega_f^+})$. Hence $ord_2(S'_m/\Omega_f^+) = ord_2 (L(E,1)/{\Omega_f^+})$, and the proof is complete.
\end{proof}

\smallskip

\begin{lem}\label{lemmaS'_m-2}
Let $E$ be a $\Gamma_0(C)$-optimal elliptic curve over $\BQ$ with $ord_2 (L(E,1)/{\Omega_f^+})=-1$. Let $m$ be an odd square-free integer greater than $1$ with $(m,C)=1$. Assume that $q \equiv 3 \mod4$ and $N_q \equiv 2 \mod 4$ for each prime $q$ dividing $m$. Then
$$
ord_2(S'_m/\Omega_f^+) = r(m)-1.
$$
\end{lem}
\begin{proof}
When $r(m)=1$, say $m=q_1$, the assertion of the lemma follows immediately from \eqref{ms_qr}. Now assume $r(m)>1$, and assume the lemma is true for all divisors $n>1$ of $m$ with $n \neq m$. Note also that $q \equiv 3 \mod4$ and $N_q \equiv 2 \mod 4$ for each prime $q$ dividing $m$, and so
$$
ord_2 \left((1+q_1)(1+q_2) \cdots (1+q_{r(m)})-a_{q_1}a_{q_2} \cdots a_{q_{r(m)}} \right)=r(m).
$$
But, by Lemma \ref{sum_Sl}, we have
\begin{equation}\label{sum_Sl-2}
\sum_{l \mid m} S_l / \Omega_f^+ = S'_m / \Omega_f^+ + \left(2 \sum_{\substack{n \mid m \\ r(n)=r(m)-1}} S'_n + 2^{2} \sum_{\substack{n \mid m \\ r(n)=r(m)-2}} S'_n + \cdots + 2^{r(m)-1} \sum_{\substack{n \mid m \\ r(n)=1}} S'_n \right) / \Omega_f^+.
\end{equation}
Suppose first that $r(m)$ is odd, by our induction hypothesis, it is easy to see that
$$
2 \sum_{\substack{n \mid m \\ r(n)=r(m)-1}} S'_n / \Omega_f^+ + 2^{r(m)-1} \sum_{\substack{n \mid m \\ r(n)=1}} S'_n  / \Omega_f^+,
$$
$$
\cdots,
$$
$$
2^{\frac{r(m)-1}{2}} \sum_{\substack{n \mid m \\ r(n)=(r(m)+1)/2}} S'_n / \Omega_f^+ + 2^{\frac{r(m)+1}{2}} \sum_{\substack{n \mid m \\ r(n)=(r(m)-1)/2}} S'_n  / \Omega_f^+
$$
are all divisible by $2^{r(m)}$, so the sum of all the terms in the second part on the right hand side of \eqref{sum_Sl-2} has order strictly greater than $r(m)-1$. Also note that $ord_2 (L(E,1)/{\Omega_f^+})=-1$, whence, it follows that $ord_2(S'_m/\Omega_f^+) = r(m)-1$.
We then suppose that $r(m)$ is even. By our induction hypothesis, it is easy to see that all
$$
2 \sum_{\substack{n \mid m \\ r(n)=r(m)-1}} S'_n / \Omega_f^+ + 2^{r(m)-1} \sum_{\substack{n \mid m \\ r(n)=1}} S'_n  / \Omega_f^+,
$$
$$
2^2 \sum_{\substack{n \mid m \\ r(n)=r(m)-2}} S'_n / \Omega_f^+ + 2^{r(m)-2} \sum_{\substack{n \mid m \\ r(n)=2}} S'_n  / \Omega_f^+,
$$
$$
\cdots,
$$
$$
2^{\frac{r(m)-2}{2}} \sum_{\substack{n \mid m \\ r(n)=(r(m)+2)/2}} S'_n / \Omega_f^+ + 2^{\frac{r(m)+2}{2}} \sum_{\substack{n \mid m \\ r(n)=(r(m)-2)/2}} S'_n  / \Omega_f^+,
$$
$$
2^{\frac{r(m)}{2}} \sum_{\substack{n \mid m \\ r(n)=r(m)/2}} S'_n / \Omega_f^+
$$
are divisible by $2^{r(m)}$. Similarly, it follows that $ord_2(S'_m/\Omega_f^+) = r(m)-1$. The proof of the lemma is complete.
\end{proof}

\smallskip

\begin{lem}\label{lemmaS'_m-3}
Let $E$ be a $\Gamma_0(C)$-optimal elliptic curve over $\BQ$, with $L(E,1) \neq 0$. Let $m$ be an odd square-free integer greater than $1$ with $(m,C)=1$. Assume that $ord_2(N_q) + ord_2 (L(E,1)/{\Omega_f^+}) > 0$ for at least one prime factor $q$ of $m$. Then
$$
ord_2(S'_m/\Omega_f^+) \geq 1.
$$
\end{lem}
\begin{proof}
The proof is similar to the above two proofs. We first note that
$$
ord_2(S'_q/\Omega_f^+)=ord_2(N_{q} L(E,1)/{\Omega_f^+}) \geq 1.
$$
The lemma then follows easily by an induction on $r$.
\end{proof}

\bigskip

\section{Period lattice and the proof of non-vanishing results}

In this section, we prove the non-vanishing results of Section 1 combining the crucial lemmas in the previous section with some elementary facts on the period lattice of elliptic curves.

When the discriminant of $E$ is negative, then $E(\BR)$ has only one real component, and so the period lattice $\fL$ of the N\'{e}ron differential on $E$ has a $\BZ$-basis of the form
$$
\left[\Omega_E^+, \frac{\Omega_E^+ + i \Omega_E^-}{2}\right],
$$
where $\Omega_E^+$ and $\Omega_E^-$ are both real, and the period lattice $\Lambda_f$ of $f$ has a $\BZ$-basis of the form
$$
\left[\Omega_f^+, \frac{\Omega_f^+ + i \Omega_f^-}{2}\right],
$$
where $\Omega_f^+$ and $\Omega_f^-$ are also both real. When the discriminant of $E$ is positive, then $E(\BR)$ has two real components, and so the period lattice $\fL$ of the N\'{e}ron differential on $E$ has a $\BZ$-basis of the form
$$
[\Omega_E^+, i \Omega_E^-],
$$
with $\Omega_f^+$ and $\Omega_E^-$ real numbers, and the period lattice $\Lambda_f$ of $f$ has a $\BZ$-basis of the form
$$
[\Omega_f^+, i \Omega_f^-],
$$
with $\Omega_f^+$ and $\Omega_f^-$ real numbers too. One can find detailed descriptions of the period lattice of elliptic curves in Cremona's book \cite[Chapter 2]{Cremona}.

Now we give the proof of our theorems. We use the same notations as before, and denote $m=M/\epsilon >0$ in what follows of this section. Moreover, note that we have assumed that the Manin constant is always odd, so we will have
$$
ord_2(L^{(alg)}(E,1)) = ord_2(L(E,1)/{\Omega_f^+}),
$$
$$
ord_2(L^{(alg)}(E^{(M)},1)) = ord_2(\sqrt{M}L(E^{(M)},1)/{\Omega_f^+})
$$
when $M$ is positive, and
$$
ord_2(L^{(alg)}(E^{(M)},1)) = ord_2(\sqrt{M}L(E^{(M)},1)/{i\Omega_f^-})
$$
when $M$ is negative in the following arguments.

\smallskip

\begin{proof}[Proof of Theorem \ref{thm1}]
Firstly, as $E[2](\BQ)=0$ and $q_1, q_2, \ldots, q_r$ are inert in $F$, then we have that $\#E(\BF_{q_i})[2]=0$, that means the order of $E(\BF_{q_i})$ must be odd, where $1 \leq i \leq r$. So $a_i$ is odd by applying $a_q=q+1-\#E(\BF_q)$, i.e. $N_{q_i}$ is odd for any $1 \leq i \leq r$. Secondly, as $E$ has negative discriminant, we can write
$$
\langle \{0,\frac{k}{m}\}, f \rangle = (s_k \Omega_f^+ + i t_k \Omega_f^-)/2
$$
for any integer $m$ coprime to $C$, where $s_k, t_k$ are integers of the same parity. Moreover, by the basic property of modular symbols, $\langle \{0,\frac{k}{m}\}, f \rangle$ and $\langle \{0,\frac{m-k}{m}\}, f \rangle$ are complex conjugate periods of $E$. Thus we obtain
$$
S'_m/\Omega_f^+=\sum_{\substack{k=1 \\ (k,m)=1}}^{(m-1)/2}  s_k.
$$
Then by Lemma \ref{lemmaS'_m}, it follows that
$$
\sum_{\substack{k=1 \\ (k,m)=1}}^{(m-1)/2} s_k
$$
is an odd integer. On the other hand, according to \eqref{ms2}, we have that
$$
\sqrt{M}L(E^{(M)},1)=\sum_{k=1}^{m} \chi(k) \langle \{0,\frac{k}{m}\}, f \rangle.
$$

When $M>0$, i.e. $\epsilon=1$, noting that $\chi(k)=\chi(m-k)$, it follows easily that
$$
\sqrt{M}L(E^{(M)},1)/\Omega_f^+ = \sum_{k=1}^{(m-1)/2} \chi(k) s_k \equiv \sum_{\substack{k=1 \\ (k,m)=1}}^{(m-1)/2} s_k \ \mod 2.
$$
The last congruence holds because $\chi(k) \equiv 1 \mod 2$ when $(k,m)=1$, and $\chi(k)=0$ when $(k,m)>1$. The assertion of the theorem now follows when $M>0$.

When $M<0$, i.e. $\epsilon=-1$, noting that $\chi(k)=-\chi(m-k)$, it follows easily that
$$
\sqrt{M}L(E^{(M)},1)/(i \Omega_f^-) = \sum_{k=1}^{(m-1)/2} \chi(k) t_k \equiv \sum_{\substack{k=1 \\ (k,m)=1}}^{(m-1)/2} t_k \equiv \sum_{\substack{k=1 \\ (k,m)=1}}^{(m-1)/2} s_k \ \mod 2.
$$
The last congruence holds because $\chi(k) \equiv 1 \mod 2$ when $(k,m)=1$, and $\chi(k)=0$ when $(k,m)>1$, and noting that $s_k, t_k$ are of the same parity. The assertion of the theorem now follows when $M<0$.

Hence
$$
ord_2(L^{(alg)}(E^{(M)},1))=0
$$
for both cases. This completes the proof of Theorem \ref{thm1}.
\end{proof}

\smallskip

The proof of Theorem \ref{thm1-1} is similar to the proof of \ref{thm1}.

\begin{proof}[Proof of Theorem \ref{thm1-1}]
Since $E$ has positive discriminant, we can write
$$
\langle \{0,\frac{k}{m}\}, f \rangle = s_k \Omega_f^+ + i t_k \Omega_f^-
$$
for any integer $m$ coprime to $C$, where $s_k, t_k$ are integers, but are independent with the ones in the above proof. Then by Lemma \ref{lemmaS'_m}, it follows that
$$
ord_2 \left(2\sum_{\substack{k=1 \\ (k,m)=1}}^{(m-1)/2} s_k \right)=ord_2(L(E,1)/{\Omega_f^+})=1.
$$
Thus
$$
\sum_{\substack{k=1 \\ (k,m)=1}}^{(m-1)/2} s_k
$$
is an odd integer. Noting that $\chi(k)=\chi(m-k)$, it follows easily that
$$
\sqrt{M}L(E^{(M)},1)/\Omega_f^+ = 2\sum_{k=1}^{(m-1)/2} \chi(k) s_k \equiv 2\sum_{\substack{k=1 \\ (k,m)=1}}^{(m-1)/2} s_k \ \mod 4.
$$
Hence
$$
ord_2(L^{(alg)}(E^{(M)},1))=1.
$$
This completes the proof of Theorem \ref{thm1-1}.
\end{proof}

\smallskip

We remark here that when the discriminant of $E$ is negative and $E[2](\BQ)=0$, we must have $ord_2(L(E,1)/{\Omega_f^+}) \geq 0$; and when the discriminant of $E$ is positive and $E[2](\BQ)=0$, we must have $ord_2(L(E,1)/{\Omega_f^+}) \geq 1$. These assertions could be easily seen from the proofs of Theorem \ref{thm1} and Theorem \ref{thm1-1}.

\smallskip

We now prove Theorem \ref{thm2} and Theorem \ref{thm2-1}.

\begin{proof}[Proof of Theorem \ref{thm2} and \ref{thm2-1}]
When $E$ has negative discriminant and $ord_2 (L(E,1)/{\Omega_f^+})+ord_2(N_q)=0$, we have that
$$
ord_2 \left(\sum_{k=1}^{(q-1)/2} s_k \right) = ord_2(N_q L(E,1)/{\Omega_f^+}) = 0.
$$
Theorem \ref{thm2} then follows by the same argument in the proof of Theorem \ref{thm1}.

When $E$ has positive discriminant and $ord_2 (L(E,1)/{\Omega_f^+})+ord_2(N_q)=1$, we have that
$$
ord_2 \left(2\sum_{k=1}^{(q-1)/2} s_k \right) = ord_2(N_q L(E,1)/{\Omega_f^+}) = 1.
$$
Thus
$$
\sum_{k=1}^{(q-1)/2} s_k
$$
is an odd integer. Theorem \ref{thm2-1} then follows by the same argument in the proof of Theorem \ref{thm1-1}.
\end{proof}

\smallskip

We now prove Theorem \ref{thm3} and Theorem \ref{thm3-1}.

\begin{proof}[Proof of Theorem \ref{thm3} and \ref{thm3-1}]
When $E$ has negative discriminant and $ord_2 (L(E,1)/{\Omega_f^+}) + ord_2(N_{q_i}) > 0$, we have that
$$
ord_2\left(\sum_{\substack{k=1 \\ (k,m)=1}}^{(m-1)/2} s_k \right) \geq 1
$$
by Lemma \ref{lemmaS'_m-3}. Thus
$$
\sum_{\substack{k=1 \\ (k,m)=1}}^{(m-1)/2} s_k
$$
is even. Then both
$$
\sum_{k=1}^{(m-1)/2} \chi(k) s_k \ \ \text{and} \ \ \sum_{k=1}^{(m-1)/2} \chi(k) t_k
$$
are even. So we have
$$
ord_2(L^{(alg)}(E^{(M)},1)) = ord_2 \left(\sum_{k=1}^{(m-1)/2} \chi(k) s_k \right) \geq 1
$$
when $M>0$, and
$$
ord_2(L^{(alg)}(E^{(M)},1)) = ord_2 \left(\sum_{k=1}^{(m-1)/2} \chi(k) t_k \right) \geq 1
$$
when $M<0$. This proves Theorem \ref{thm3}.

When $E$ has positive discriminant, of course we have
$$
ord_2(L^{(alg)}(E^{(M)},1)) = ord_2 \left(2\sum_{k=1}^{(m-1)/2} \chi(k) s_k \right) \geq 1
$$
when $M>0$, and
$$
ord_2(L^{(alg)}(E^{(M)},1)) = ord_2 \left(2\sum_{k=1}^{(m-1)/2} \chi(k) t_k \right) \geq 1
$$
when $M<0$. This proves Theorem \ref{thm3-1}.
\end{proof}

\bigskip

\section{2-Selmer groups}

In this section, we shall prove Theorem \ref{thm1-bsd} and \ref{thm1-1-bsd}, say that there are many explicit quadratic twists of a large class of elliptic curves satisfying the $2$-part of the Birch and Swinnerton-Dyer conjecture. The main tools in this section are some results of Mazur and Rubin \cite{Mazur} that compare Selmer groups of $E$ and $E^{(M)}$ by different local conditions. Here we follow the notations in \cite{Mazur}.

Let $E$ is an elliptic curve over $\BQ$, and let $\Delta_E$ denote the discriminant of $E$. Let $K$ denote the quadratic field $\BQ(\sqrt{M})$ with integer $M \equiv 1 \mod 4$. For every place $v$ of $\BQ$, let $H^1_f(\BQ_v, E[2])$ denote the image of the Kummer map
$$
E(\BQ_v)/2E(\BQ_v) \longrightarrow H^1(\BQ_v, E[2]).
$$
Let $E_N(\BQ_v) \subset E(\BQ_v)$ denote the image of the norm map $E(K_w) \rightarrow E(\BQ_v)$ for any choice of $w$ above $v$, and define
$$
\delta_v(E, K/\BQ):=\dim_{\BF_2}(E(\BQ_v)/E_N(\BQ_v)).
$$
Let $Sel_2(E)$ and $Sel_2(E^{(M)})$ denote the $2$-Selmer group of $E$ and $E^{(M)}$ over $\BQ$, respectively. The $2$-Selmer group $Sel_2(E) \subset H^1(\BQ, E[2])$ is the $\BF_2$-vector space defined by the following exact sequence
$$
0 \to Sel_2(E) \to H^1(\BQ, E[2]) \to \bigoplus_v H^1(\BQ_v, E[2])/H^1_f(\BQ_v, E[2]).
$$
Since there is a natural identification of Galois modules $E[2]=E^{(M)}[2]$, which allows us to view $Sel_2(E), Sel_2(E^{(M)}) \subset H^1(\BQ, E[2])$, defined by different sets of local conditions.

Let $T$ denote a finite set of places of $\BQ$, and define the sum of the localization maps as following
$$
\loc_T : H^1(\BQ, E[2]) \longrightarrow \bigoplus_{v \in T} H^1(\BQ_v, E[2]).
$$
Define strict and relaxed $2$-Selmer groups $\CS_T \subset \CS^T \subset H^1(\BQ, E[2])$ by the exactness of
$$
0 \to \CS^T \to H^1(\BQ, E[2]) \to \bigoplus_{v \notin T} H^1(\BQ_v, E[2])/H^1_f(\BQ_v, E[2]),
$$
$$
0 \to \CS_T \to \CS^T \xrightarrow{\loc_T} \bigoplus_{v \in T} H^1(\BQ_v, E[2]).
$$
Then by the above definition we have $\CS_T \subset Sel_2(E) \subset \CS^T$.

\smallskip

The following two lemmas \cite[Lemma 2.10, 2.11]{Mazur} are criteria for equality and transversality of local conditions after twist, which are very crucial when we bound the $2$-Selmer group of $E^{(M)}$.

\begin{lem}[Mazur-Rubin]\label{lemma2.10M&R}
If at least one of the following conditions holds:
\begin{enumerate}
  \item $v$ splits in $K$, or
  \item $v \nmid 2\infty$ and $E(\BQ_v)[2]=0$, or
  \item $E$ has multiplicative reduction at $v$, $K/\BQ$ is unramified at $v$, and $ord_v(\Delta_E)$ is odd, or
  \item $v$ is real and $(\Delta_E)_v < 0$, or
  \item $v$ is a prime where $E$ has good reduction and $v$ is unramified in $K/\BQ$,
\end{enumerate}
then $H^1_f(\BQ_v, E[2])=H^1_f(\BQ_v, E^{(M)}[2])$ and $\delta_v(E, K/\BQ)=0$.
\end{lem}

\smallskip

\begin{lem}[Mazur-Rubin]\label{lemma2.11M&R}
If $v \nmid 2\infty$, $E$ has good reduction at $v$, and $v$ is ramified in $K/\BQ$, then
$H^1_f(\BQ_v, E[2]) \cap H^1_f(\BQ_v, E^{(M)}[2])=0$ and $\delta_v(E, K/\BQ)=\dim_{\BF_2}(E(\BQ_v)[2])$.
\end{lem}

By Lemma \ref{lemma2.10M&R}, $H^1_f(\BQ_v, E[2])=H^1_f(\BQ_v, E^{(M)}[2])$ if $v \notin T$, then we have $\CS_T \subset Sel_2(E^{(M)}) \subset \CS^T$. By Lemma \ref{lemma2.11M&R}, we have $Sel_2(E) \cap Sel_2(E^{(M)})=\CS_T$ and $Sel_2(E) + Sel_2(E^{(M)}) \subset \CS^T$. We then have the following two parallel results.

\begin{prop}
Assume that $Sel_2(E)=0$, and in addition to the hypotheses of Theorem \ref{thm1}, we also suppose that the bad primes of $E$ all split in $K$. Then we have $\dim_{\BF_2}(Sel_2(E^{(M)}))=0$.
\end{prop}

\begin{proof}
If $v$ is a bad prime of $E$, then it splits in $K$ by the assumption. Then by Lemma \ref{lemma2.10M&R}, we have that
$H^1_f(\BQ_v, E[2])=H^1_f(\BQ_v, E^{(M)}[2])$ and $\delta_v(E, K/\BQ)=0$. This also holds when $v=2$.

If $v=2$ is a good prime of $E$, obviously it is unramified in $K$, since $M \equiv 1 \mod 4$. Then again by Lemma \ref{lemma2.10M&R}, we still have that $H^1_f(\BQ_v, E[2])=H^1_f(\BQ_v, E^{(M)}[2])$ and $\delta_v(E, K/\BQ)=0$. This also holds for those odd good primes which are unramified in $K$.

If $v$ is a good prime of $E$ and ramified in $K$, of course we have $v \mid M$, whence we have that $\delta_v(E, K/\BQ)=\dim_{\BF_2}(E(\BQ_v)[2])$ by Lemma \ref{lemma2.11M&R}. In this case, it is easy to get that $2 \nmid \#E(\BQ_v)[2]$, this is because $v$ is inert in $F$ ($F$ is the field obtained by adjoining to $\BQ$ one fixed root of the $2$-division polynomial of $E$), and $E[2](\BQ) = 0$. Hence $\delta_v(E, K/\BQ)=0$. It follows that $H^1_f(\BQ_v, E[2])=H^1_f(\BQ_v, E^{(M)}[2])=0$.

If $v$ is the infinite place, when the discriminant of $E$ is negative, it follows that $H^1_f(\BQ_v, E[2])=H^1_f(\BQ_v, E^{(M)}[2])$ and $\delta_v(E, K/\BQ)=0$ by Lemma \ref{lemma2.10M&R}.

Now the assertion of the proposition follows when the discriminant of $E$ is negative. This is because the local conditions $H^1_f(\BQ_v, E[2]), H^1_f(\BQ_v, E^{(M)}[2]) \subset H^1(\BQ_v, E[2])$ now coincide for all places of $\BQ$, then the local conditions defining the $2$-Selmer groups of $E/\BQ$ and $E^{(M)}/\BQ$ agree. It then follows that $\dim_{\BF_2}(Sel_2(E^{(M)}))=\dim_{\BF_2}(Sel_2(E))=0$.
\end{proof}

\smallskip

\begin{prop}
Assume that $Sel_2(E)=0$, and in addition to the hypotheses of Theorem \ref{thm1-1}, we also suppose that the bad primes of $E$ all split in $K$. Then we have $\dim_{\BF_2}(Sel_2(E^{(M)}))=0$.
\end{prop}

\begin{proof}
When the discriminant of $E$ is positive, the local conditions defining the $2$-Selmer groups of $E/\BQ$ and $E^{(M)}/\BQ$ now coincide for all places but the infinite place. Therefore we can define $T=\{\infty\}$, and $S_T=0$ as $Sel_2(E)=0$. By \cite[Lemma 3.2]{Mazur}, we have that $\dim_{\BF_2}(S^T)=\dim_{\BF_2}(E(\BR)/2E(\BR))$, which is $1$ since the discriminant of $E$ is positive. So $\dim_{\BF_2}(Sel_2(E^{(M)}))$ is equal or less than $1$. $\dim_{\BF_2}(Sel_2(E^{(M)}))$ is exactly $0$ by Cassels-Tate pairing, because we have proved that $\Sha(E^{(M)}/\BQ)$ is finite and $rank(E^{(M)}/\BQ)=0$ in Theorem \ref{thm1-1}. The assertion of the proposition then follows.
\end{proof}

\medskip

In order to understand the $2$-part of the Birch and Swinnerton-Dyer conjecture for $E^{(M)}$, we have to understand how the $2$-part of the Tamagawa factors of $E^{(M)}$ vary for primes. We assume once again that $M \equiv 1 \mod 4$ is an arbitrary square-free integer with $(M, C)=1$. Note that $E^{(M)}$ has bad additive reduction at all primes dividing $M$. Write $c_q(E^{(M)})$ for the Tamagawa factor of $E^{(M)}$ at a finite odd prime $q$. We then have the following lemma, and one can find a detailed discussion in \cite[\S7]{Coates2}.

\begin{lem}\label{Tam_c_q}
For any odd prime $q \mid M$, we have that
$$
ord_2(c_q(E^{(M)})) = ord_2(\# E(\BQ_q)[2]).
$$
\end{lem}

\smallskip

\begin{prop}
Let $E$ be an elliptic curve over $\BQ$, with $E[2](\BQ)=0$. Let $M$ be any integer of the form $M=\pm q_1q_2 \cdots q_{r}$, and prime to the conductor of $E$, where $r \geq 1$, $q_1, \ldots, q_{r}$ are arbitrary distinct odd primes which are inert in the field $F$. Then we have
$$
ord_2(c_{q_i}(E^{(M)}))=0
$$
for any integer $1 \leq i \leq r$.
\end{prop}
\begin{proof}
Note that $E[2](\BQ)=0$ and $q_i$ ($1 \leq i \leq r$) is inert in $F$, so $\# E(\BQ_{q_i})[2]$ must be an odd integer. It then follows easily by the above lemma.
\end{proof}

\smallskip

\begin{prop}
Let $E$ be an elliptic curve over $\BQ$, with conductor $C$. Let $M$ be any square-free integer. If all the bad primes of $E$ split in $\BQ(\sqrt{M})$, then we have
$$
c_p(E^{(M)})=c_p(E)
$$
for any bad prime $p$.
\end{prop}
\begin{proof}
Since $M$ is a square modulo $p$ for any $p \mid C$, then there is an isomorphism from $E^{(M)}$ to $E$ which identifies $E^{(M)}(\BQ_p)$ with $E(\BQ_p)$. By the definition of $c_p(E)$, it follows that $c_p(E^{(M)})=c_p(E)$.
\end{proof}

\smallskip

We now give the proof of Theorem \ref{thm1-bsd} and Theorem \ref{thm1-1-bsd}.

\begin{proof}[Proof of Theorem \ref{thm1-bsd} and \ref{thm1-1-bsd}]
Under the assumption of the $2$-part of Birch and Swinnerton-Dyer conjecture for $E$, we have that $ord_2(c_p(E))=0$. Now the two theorems follow by combining the results of Theorem \ref{thm1}, Theorem \ref{thm1-1}, and the above results.
\end{proof}

\smallskip

\noindent
\emph{Remark.} Under the hypothesises of Theorem \ref{thm1} or \ref{thm1-1}, and assuming that the $2$-part of Birch and Swinnerton-Dyer conjecture holds for $E$. Then there are infinitely many elliptic curves satisfying the $2$-part of the Birch and Swinnerton-Dyer conjecture. This is because we can always choose $M$ to make all the bad primes of $E$ split in $\BQ(\sqrt{M})$.

\bigskip

\section{Quadratic twists of Neumann--Setzer elliptic curves}

In this section, we shall take the Neumann--Setzer elliptic curves as an example of Theorem \ref{thm2}, and verify the $2$-part of Birch and Swinnerton-Dyer conjecture for a family of quadratic twist of the curves.

Let $p$ be a prime of the form $u^2+64$ for some integer $u$, which is congruent to $1$ modulo $4$. According to Neumann \cite{Neu1}\cite{Neu2} and Setzer \cite{Set}, there are just two elliptic curves of conductor $p$, up to isomorphism, namely,
\begin{eqnarray*}
A :& y^2 + xy &= x^3 + \frac{u-1}{4}x^2 + 4x + u, \\
A':& y^2 + xy &= x^3 -\frac{u-1}{4}x^2 -x.
\end{eqnarray*}
The curves $A$ and $A'$ are $2$-isogenous and both of which have Mordell--Weil groups which are finite of order $2$. The discriminant of $A$ is $-p^2$, and the discriminant of $A'$ is $p$. We denote $F$ and $F'$ to be the 2-division fields of $A$ and $A'$, respectively. It is easy to show that
$$
\BQ(A[2])=\BQ(i), \ \BQ(A'[2])=\BQ(\sqrt{p}).
$$
Let $X_0(p)$ be the modular curve of level $p$, and there is a non-constant rational map $X_0(p) \to A$, of which the modular parametrization is $\Gamma_0(p)$-optimal by Mestre and Oesterl\'e \cite{MO}.

\medskip

\subsection{Classical 2-descents}

In order to carry out the 2-descent, we must work with a new equation for $A$ and its twists. Making change of variables, we obtain the following equation for $A$:
$$
Y^2 = X^3 - 2uX^2 + pX.
$$
Let $M$ be any square-free integer $\neq1$, and let $A^{(M)}$ be the twist of $M$ by the quadratic extension $\BQ(\sqrt{M})/\BQ$. Then the curve $A^{(M)}$ will have equation
$$
A^{(M)}: y^2 = x^3 - 2uM x^2 + pM^2 x.
$$
and, dividing this curve by the subgroup generated by the point $(0,0)$, we obtain the new curve
$$
A^{'(M)}: y^2 = x^3 + 4uM x^2 - 256M^2 x.
$$
Explicitly, the isogenies between these two curves, are given by
$$
\phi:A^{(M)}\rightarrow A^{'(M)}, \  (x,y) \mapsto \left(\frac{y^2}{x^2},\frac{y(pM^2-x^2)}{x^2} \right);
$$
$$
\hat{\phi}:A^{'(M)}\rightarrow A^{(M)}, \  (x,y) \mapsto \left(\frac{y^2}{4x^2},\frac{y(-256M^2-x^2)}{8x^2} \right).
$$
We write $S^{(\phi)}(A^{(M)})$ and $S^{(\hat{\phi})}(A^{'(M)})$ for the classical Selmer groups of the isogenies $\phi$ and $\hat{\phi}$, which can be described explicitly as follows. Let $V$ denote the set of all places of $\BQ$, and let $T_M$ be the set of primes dividing $2pM$. Let $\BQ(2, M)$ be the subgroup of $\BQ^\times/(\BQ^\times)^2$ consisting of all elements with a representative which has even order at each prime number not in $T_M$. Writing
\begin{equation}\label{Cd}
C_{d}: dw^2 = 4p - \left(\frac{M}{d}z^2 - 2u\right)^2,
\end{equation}
then $S^{(\phi)}(A^{(M)})$ can be naturally identified with the subgroup of all $d$ in $\BQ(2,M)$ such that $C_d(\BQ_v)$ is non-empty for $v=\infty$ and $v$ dividing $2pM$.
Similarly, writing
\begin{equation} \label{C'd}
C'_d : dw^2 = 64p^3 + p\left(\frac{M}{d}z^2 - up\right)^2,
\end{equation}
then $S^{(\hat{\phi})}(A^{'(M)})$ can be naturally identified with the subgroup of all $d$ in $\BQ(2,M)$ such that ${C'}_{d}(\BQ_v)$ is non-empty for $v=\infty$ and $v$ dividing $2pM$. Note that $-1 \in S^{(\phi)}(A^{(M)})$ because it is the image of the point $(0, 0)$ in $A^{'(M)}(\BQ)$, and similarly $p \in S^{(\hat{\phi})}(A^{'(M)})$ (see Proposition 4.9 of \cite{Sil}).

\medskip

If $D$ is any odd square-free integer, we define $D_{+}$ (resp. $D_{-}$) to be the product of the primes dividing $D$, which are $\equiv 1 \mod 4$ (resp. which are $\equiv 3 \mod 4$). In what follows, we shall always assume that $M$ is an odd square-free integer which is prime to $p$, and let $R$ denote the product of the prime factors of $M$ which are inert in the field $\BQ(\sqrt{p})$, and let $N$ denote the product of prime factors of $M$ which split in the field $\BQ(\sqrt{p})$, and let $\left(\frac{\cdot}{q}\right)$ be the Jacobi symbol. We will then write $M = \epsilon RN$, where $\epsilon = \pm 1$.

\medskip

\begin{prop}\label{DescentCd}
Let $M$ be an odd square-free integer which is prime to $p$. Then $S^{(\phi)}(A^{(M)})$ consists of the classes in $\BQ(2,M)$ represented by all integers $d, -d$ satisfying the following conditions:
\begin{enumerate}
\item $d$ divides $N$;
\item $\left(\frac{d}{q}\right)=1$ for all primes $q$ dividing $M_+/(M_+,d)$, and $\left(\frac{M/d}{q}\right)=\left(\frac{2u+2a}{q}\right)$ for all primes $q$ dividing $(N_+,d)$, where $a$ is an integer satisfying $a^2 \equiv p \mod q$.
\end{enumerate}
\end{prop}

\begin{proof}
We recall that $C_d$ denotes the curve \eqref{Cd}. We see immediately that $C_d(\BR) \neq \emptyset$.

If $2$ divides $d$, a point on $C_d$ with coordinates in $\BQ_2$ must have coordinates in $\BZ_2$, whence it follows easily that $C_d(\BQ_2) = \emptyset$. If $p$ divides $d$, a point on $C_d$ with coordinates in $\BQ_{p}$ must have coordinates in $\BZ_{p}$, whence it follows easily that $C_d(\BQ_{p}) = \emptyset$. So next we need only to consider the cases when $2 \nmid d$ and $p \nmid d$.

We claim that
\begin{equation} \label{C2}
C_d(\BQ_2) \neq \emptyset
\end{equation}
is always true for any odd integer $d$. Note that \eqref{Cd} has a solution in $\BQ_2$ with $w=0$ for any $M/d$. So our claim follows.

\medskip

We now determine when
\begin{equation} \label{Cp}
C_d(\BQ_{p}) \neq \emptyset.
\end{equation}
We shall prove that \eqref{Cp} is true if and only if $\left(\frac{d}{p}\right)=1$. Note first that \eqref{Cd} has a solution in $\BQ_p$ with $z=0$ if and only if $\left(\frac{d}{p}\right)=1$, and there is no solution when $w=0$. We then put $w=p^{-m}w_1, z=p^{-n}z_1$, where $m, n > 0$, and $w_1, z_1$ are in $\BZ_{p}^\times$. Then a necessary condition for a solution is that $m=2n$, and we then obtain the new equation
$$
dw_1^2=4p^{4n+1}-\left(\frac{M}{d}z_1^2-2up^{2n}\right)^2,
$$
which is soluble modulo $p$ if and only if $\left(\frac{-d}{p}\right)=1$, i.e. $\left(\frac{d}{p}\right)=1$. Next we put $w=p^m w_1, z =p^n z_1$, where $m, n \geq 0$, and $w_1, z_1$ are in $\BZ_p^\times$, the equation then becomes
$$
dw_1^2=\frac{4p-\left(p^{2n}\frac{M}{d}z_1^2-2u\right)^2}{p^{2m}}.
$$
It follows easily that we must have $m=0$ and $n \geq 0$. For $m=0$ and $n=0$, the equation becomes
$$
dw_1^2=4p-\left(p^{2n}\frac{M}{d}z_1^2-2u\right)^2.
$$
Taking the above equation modulo $p$, and then we have that it is soluble in $\BQ_p$ if and only if $\left(\frac{d}{p}\right)=1$. This proves our claim for \eqref{Cp}.

\medskip

We now determine when
\begin{equation} \label{Cq}
C_d(\BQ_q) \neq \emptyset,
\end{equation}
where $q$ is a prime factor of $M$. Assume first that $q$ divides $d$. We claim that \eqref{Cq} is always true when $q \mid N_-$ and is true if and only if $\left(\frac{M/d}{q}\right)=\left(\frac{2u+2a}{q}\right)$ when $q \mid N_+$, where $a$ is an integer satisfying $a^2 \equiv p \mod q$. Indeed, a point on $C_d$ with coordinates in $\BQ_q$ must have coordinates in $\BZ_q$. Taking the equation of $C_d$ modulo $q$, it then becomes
$$
4p - \left(\frac{M}{d}z^2 - 2u\right)^2  \equiv 0 \ \mod q.
$$
It is easy to see that a necessary condition for the solubility is $\left(\frac{p}{q}\right)=1$. We assume this and $a^2 \equiv p \mod q$, then the equation becomes
$$
\frac{M}{d}z^2 \equiv 2u \pm 2a \ \mod q.
$$
Note that $(2u+2a)(2u-2a) \equiv -256 \mod q$ and $\left(\frac{-1}{q}\right)=-1$ when $q \equiv 3 \mod 4$, so \eqref{Cq} will always be true when $q$ divides $N_-$, and it will be true when $q$ divides $N_+$ if and only if $\left(\frac{M/d}{q}\right)=\left(\frac{2u+2a}{q}\right)$. This proves our claim. Now assume that $q$ does not divide $d$. We claim that \eqref{Cq} is always true when $q \mid M_-$ and is true if and only if  $\left(\frac{d}{q}\right)=1$ when $q \mid M_+$. Indeed, if $\left(\frac{d}{q}\right)=1$, the congruence given by putting $z=0$ in the equation of $C_d$ modulo $q$ is clearly soluble, and this gives a point on $C_d$ with coordinates in $\BZ_q$. Conversely, if there is a point on $C_d$ with coordinates in $\BZ_q$, it follows immediately that $\left(\frac{d}{q}\right)=1$. On the other hand, if there is a point $(w, z)$ on $C_d$ with non-integral coordinates, we can write $w=q^{-m}w_1, z=q^{-n}z_1$ with $m, n>0$ and $w_1, z_1 \in \BZ_q^\times$. It then follows that $m=2n-1$ and the equation becomes
$$
dw_1^2=4pq^{2m} - \left(\frac{M}{qd}z_1^2 - 2uq^m\right)^2.
$$
Taking this last equation modulo $q$, we conclude that $\left(\frac{d}{q}\right)=\left(\frac{-1}{q}\right)$. Our claim then follows.

Putting together all of the above results, the proof of Proposition \ref{DescentCd} is complete.
\end{proof}

\bigskip

\begin{prop}\label{DescentC'd}
Let $M$ be an odd square-free integer which is prime to $p$. Then $S^{(\hat{\phi})}(A^{'(M)})$ consists of the classes in $\BQ(2,M)$ represented by all integers $d, pd$ satisfying the following conditions:
\begin{enumerate}
\item $d$ divides $M_+$ and $d>0$;
\item $d \equiv 1 \mod 4$ when $M \equiv 1 \mod 4$, and $d \equiv 1 \mod 8$ when $M \equiv 3 \mod 4$;
\item $\left(\frac{d}{q}\right)=1$ for all primes $q$ dividing $N/(N,d)$, and $\left(\frac{M/d}{q}\right)=\left(\frac{u+8b}{q}\right)$ for all primes $q$ dividing $(N_+,d)$, where $b$ is an integer satisfying $b^2 \equiv -1 \mod q$.
\end{enumerate}
\end{prop}

\begin{proof}
We recall that $C'_d$ denotes the curve \eqref{C'd}. It is clear that $C'_d(\BR) \neq \emptyset$ if and only if $d>0$.

If $2$ divides $d$, a point on $C'_d$ with coordinates in $\BQ_2$ must have coordinates in $\BZ_2$, whence it follows easily that $C'_d(\BQ_2) = \emptyset$. So next we need only to consider the case when $2 \nmid d$.

We claim that
\begin{equation} \label{C'2}
C'_d(\BQ_2) \neq \emptyset
\end{equation}
is true if and only if $d \equiv 1 \mod4$ when $M \equiv 1 \mod 4$, and $d \equiv 1 \mod 8$ when $M \equiv 3 \mod 4$. Note first that \eqref{C'd} has a solution in $\BQ_2$ with $z=0$ if and only if $d \equiv 1 \mod 8$, \eqref{C'd} has no solution in $\BQ_2$ with $w=0$. Put $w=2^{-m}w_1, z=2^{-n}z_1$, where $m, n > 0$, and $w_1, z_1$ are in $\BZ_2^\times$. Then a necessary condition for a solution is that $m=2n$, and we then obtain the new equation
$$
dw_1^2=2^{4n+6}p^3+p\left(\frac{M}{d}z_1^2-2^{2n}up\right)^2.
$$
Taking the above equation modulo $8$, it follows that it has a solution in $\BQ_2$ if and only if $d \equiv 1 \mod 8$. Next we put $w=2^m w_1, z = 2^n z_1$, where $m, n \geq 0$, and $w_1, z_1$ are in $\BZ_2^\times$, the equation then becomes
$$
dw_1^2=\frac{64p^3+p\left(2^{2n}\frac{M}{d}z_1^2-up\right)^2}{2^{2m}}.
$$
It follows easily that we have either $m=0$ and $n\geq1$, or $m\geq1$ and $n=0$. For $m=0$ and $n\geq1$, the equation becomes
$$
dw_1^2=64p^3+p\left(2^{2n}\frac{M}{d}z_1^2-up\right)^2.
$$
By taking the above equation modulo $8$, we have that it is soluble in $\BQ_2$ if and only if $d \equiv 1 \mod 8$. For $m\geq1$ and $n=0$, the equation becomes
$$
dw_1^2=\frac{64p^3+p\left(\frac{M}{d}z_1^2-up\right)^2}{2^{2m}}.
$$
When $m=1$, necessarily we have that $\frac{M}{d}z_1^2-up \equiv 2 \mod 4$, i.e. $M/d \equiv 3 \mod 4$, implying $d \equiv 1 \mod 8$. When $m=2$, necessarily we have that $\frac{M}{d}z_1^2-up \equiv 4 \mod 8$, i.e. $M/d \equiv 5 \mod 8$ when $p \equiv 1 \mod 16$ and $M/d \equiv 1 \mod 8$ when $p \equiv 9 \mod 16$, implying $d \equiv 5 \mod 8$. When $m=3$, necessarily we have that $\ord_2\left(\frac{M}{d}z_1^2-up\right) \geq 4$, i.e. $M/d \equiv 1 \mod 8$ when $p \equiv 1 \mod 16$ and $M/d \equiv 5 \mod 8$ when $p \equiv 9 \mod 16$, implying $d \equiv 1 \mod 4$. When $m\geq4$, necessarily we have that $\frac{M}{d}z_1^2-up \equiv 8 \mod 16$ and $\ord_2\left(p^3+p\left(\frac{\frac{M}{d}z_1^2-up}{8}\right)^2\right)=2m-6$, but $\ord_2\left(p^3+p\left(\frac{\frac{M}{d}z_1^2-up}{8}\right)^2\right)=1$ as $\left(\frac{\frac{M}{d}z_1^2-up}{8}\right)^2 \equiv 1 \mod 8$, which is a contradiction. Combining those cases above, we can see that \eqref{C'd} is soluble in $\BQ_2$ if and only if $d \equiv 1 \mod4$ when $M \equiv 1 \mod 4$, and $d \equiv 1 \mod 8$ when $M \equiv 3 \mod 4$. This proves our claim for \eqref{C'2}.

\medskip

We now determine when
\begin{equation} \label{C'p}
C'_d(\BQ_{p}) \neq \emptyset.
\end{equation}
We shall prove that \eqref{C'p} is always true. When $p$ divides $d$, we put $d=pd_1$, then \eqref{C'd} becomes
$$
d_1w^2=64p^2+\left(\frac{M}{pd_1}z^2-up\right)^2.
$$
Note first that \eqref{C'd} has no solution in $\BQ_p$ with $wz=0$. We put $w=p^{-m} w_1, z=p^{-n+1} z_1$, where $m, n > 0$, and $w_1, z_1$ are in $\BZ_p^\times$. Then a necessary condition for a solution is that $m=2n-1$, and we then obtain the new equation
$$
d_1w_1^2=64p^{2m+2}+\left(\frac{M}{d_1}z_1^2-up^{m+1}\right)^2.
$$
Taking the above equation modulo $p$, it follows that it has a solution in $\BQ_p$ if and only if $\left(\frac{d_1}{p}\right)=1$. Next we put $w=p^m w_1, z=p^{n+1} z_1$, where $m, n \geq 0$, and $w_1, z_1$ are in $\BZ_p^\times$, the equation then becomes
$$
d_1w_1^2=\frac{64+\left(p^{2n}\frac{M}{d_1}z_1^2-u\right)^2}{p^{2m-2}}.
$$
It follows that we must have $m\geq1, n=0$. Taking $m=1$ and $n=0$, the equation becomes
$$
d_1w_1^2=64+\left(\frac{M}{d_1}z_1^2-u\right)^2.
$$
Taking the above equation modulo $p$, we have $d_1w_1^2 \equiv \frac{M}{d_1}z_1^2 \left(\frac{M}{d_1}z_1^2-2u\right) \mod p$, which is always soluble with proper $z$ and $w$. That means \eqref{C'p} is always true when $p \mid d$. Now we assume that $p$ does not divide $d$. We see that \eqref{C'p} will be true if and only if $C'_d$ has a point with coordinates in $\BZ_p$. Next we put $w=p^m w_1, z=p^n z_1$, where $m, n \geq 0$, and $w_1, z_1$ are in $\BZ_p^\times$, the equation then becomes
$$
dw_1^2=\frac{p^{4n+1}\frac{M^2}{d^2}z_1^4-2up^{2n+2}\frac{M}{d}z_1^2+p^4}{p^{2m}}.
$$
It follows that we have either $m=2$ and $n\geq1$, or $m\geq3$ and $n=1$. Taking $m=2$ and $n=1$, and taking the above equation modulo $p$, we have that $dw_1^2 \equiv 1-2u\frac{M}{d}z_1^2 \mod p$, which is always soluble with proper $z$ and $w$. That means \eqref{C'p} is always true when $p \nmid d$. Our claim for \eqref{C'p} then follows.

\medskip

Finally, we must determine when
\begin{equation} \label{C'q}
C'_d(\BQ_q) \neq \emptyset,
\end{equation}
where $q$ is a prime factor of $M$. Assume first that $q$ divides $d$. We claim that \eqref{C'q} is true if and only if $\left(\frac{M/d}{q}\right)=\left(\frac{u+8b}{q}\right)$, where $b$ is an integer satisfying $b^2 \equiv -1 \mod q$. Indeed, a point on $C'_d$ with coordinates in $\BQ_q$ must have coordinates in $\BZ_q$. Taking the equation of $C'_d$ modulo $q$, it then becomes
$$
(\frac{M}{d}z^2-up)^2 \equiv -64p^2 \ \mod q.
$$
A necessary condition for a solution is that $q \equiv 1 \mod 4$. We now assume this condition and $b^2 \equiv -1 \mod q$, then the above equation becomes
$$
\frac{M}{d}z^2 \equiv up \pm 8pb \ \mod q,
$$
Note that $(up+8pb)(up-8pb) \equiv p^3 \mod q$, so \eqref{C'q} will always be true when $\left(\frac{p}{q}\right)=-1$, i.e. $q$ divides $R_+$, and it will be true when $q$ divides $N_+$ if and only if $\left(\frac{M/d}{q}\right)=\left(\frac{u+8b}{q}\right)$. This proves our claim. Now assume that $q$ does not divide $d$. We claim that \eqref{C'q} is always true when $q \mid R$ and is true if and only if $\left(\frac{d}{q}\right)=1$ when $q \mid N$. Indeed, if $\left(\frac{d}{q}\right)=1$, the congruence given by putting $z=0$ in the equation of $C'_d$ modulo $q$ is clearly soluble, and this gives a point on $C'_d$ with coordinates in $\BZ_q$. Conversely, if there is a point on $C'_d$ with coordinates in $\BZ_q$, it follows immediately that $\left(\frac{d}{q}\right)=1$. On the other hand, if there is a point $(w, z)$ on $C'_d$ with non-integral coordinates, we can write $w=q^{-m}w_1, z=q^{-n}z_1$ with $m, n>0$ and $w_1, z_1 \in \BZ_q^\times$. It then follows that $m=2n-1$ and the equation becomes
$$
dw_1^2=64p^3 q^{2m} + p\left(\frac{M}{qd}z_1^2 - upq^m\right)^2.
$$
Taking this last equation modulo $q$, we conclude that $\left(\frac{d}{q}\right)=\left(\frac{p}{q}\right)$. Our claim then follows.

Putting together all of the above results, the proof of Proposition \ref{DescentC'd} is complete.

\end{proof}

\medskip

We now give some consequences of Propositions \ref{DescentCd} and \ref{DescentC'd}. Assume for the rest of this paragraph that $M$ is a square-free integer, prime to $p$, with $M \equiv 1\mod 4$. In particular, it follows that the curve $A^{(M)}$ always has good reduction at $2$. When $M > 0$, its $L$-function has global root number $+1$ (reps. $-1$) if $\left(\frac{M}{p}\right)=1$ (resp. $\left(\frac{M}{p}\right)=-1$), when $M < 0$, its $L$-function has global root number $+1$ (reps. $-1$) if $\left(\frac{M}{p}\right)=-1$ (resp. $\left(\frac{M}{p}\right)=1$). We write $S^{(2)}(A^{(M)})$ for the classical Selmer group of $A^{(M)}$ for the endomorphism given by multiplication by $2$. Now it is easily seen that we have an exact sequence
$$
0 \to \frac{A'^{(M)}(\BQ)[\hat{\phi}]}{\phi(A^{(M)}(\BQ)[2])} \to S^{(\phi)}(A^{(M)}) \to S^{(2)}(A^{(M)}) \to S^{(\hat{\phi})}(A'^{(M)}).
$$
Define $\mathfrak S^{(\phi)}(A^{(M)})$, $\mathfrak S^{(\hat{\phi})}(A'^{(M)})$, $\mathfrak S^{(2)}(A^{(M)})$ and $\mathfrak S^{(2)}(A'^{(M)})$ to be the quotients of $S^{(\phi)}(A^{(M)})$, $S^{(\hat{\phi})}(A'^{(M)})$, $S^{(2)}(A^{(M)})$ and $S^{(2)}(A'^{(M)})$ by the images of the torsion subgroups of $A'^{(M)}(\BQ)$, $A^{(M)}(\BQ)$, $A^{(M)}(\BQ)$ and $A'^{(M)}(\BQ)$, respectively. By the fact that the $2$-primary subgroups of
$A'^{(M)}(\BQ)$ and $A^{(M)}(\BQ)$ are both just of order $2$, whence it follows easily that we have the exact sequence
\begin{equation}\label{exactA}
0 \to \mathfrak S^{(\phi)}(A^{(M)}) \to \mathfrak S^{(2)}(A^{(M)}) \to S^{(\hat{\phi})}(A'^{(M)}),
\end{equation}
\begin{equation}\label{exactA'}
0 \to \mathfrak S^{(\hat{\phi})}(A'^{(M)}) \to \mathfrak S^{(2)}(A'^{(M)}) \to S^{(\phi)}(A^{(M)}).
\end{equation}
Note also that the parity theorem of the Dokchitser brothers \cite{Dok} shows that $\mathfrak S^{(2)}(A^{(M)})$ has even or odd $\mathbb{F}_2$-dimension according as the root number is $+1$ or $-1$.

We now let $r(M)$, $k(M)$, $r_+(M)$, $r_-(M)$, $k_+(M)$, $k_-(M)$ denote the number of prime factors of $R$, $N$, $R_+$, $R_-$, $N_+$, $N_-$, respectively, and prove the following results.

\smallskip

\begin{cor}\label{DescentCor1}
Assume that $M = \epsilon R_-\equiv 1 \mod 4$, where $\epsilon=\pm 1$. Then $\mathfrak S^{(2)}(A^{(M)}) = 0$.
\end{cor}
\begin{proof}
Indeed Proposition \ref{DescentCd} shows that, in this case, we have $\mathfrak S^{(\phi)}(A^{(M)})=0$, and Proposition \ref{DescentC'd} shows that $S^{(\hat{\phi})}(A'^{(M)})$ has order $2$, whence the assertion follows from the exact sequence \eqref{exactA}, and the fact that $\mathfrak S^{(2)}(A^{(M)})$ must have even $\mathbb{F}_2$-dimension. Note that, of course, $r_-(M)$ has to be even when $\epsilon=1$, and $r_-(M)$ has to be odd when $\epsilon=-1$.
\end{proof}

\smallskip

\begin{cor}\label{DescentCor2}
Assume that $M = N_- \equiv 1 \mod 4$. Then $\mathfrak S^{(\phi)}(A^{(M)})$ has exact order $2^{k_-(M)}$, and $\mathfrak S^{(2)}(A^{(M)})$ has exact order $2^{k_-(M)}$.
\end{cor}
\begin{proof}
The first assertion is clear from Proposition \ref{DescentCd}. Proposition \ref{DescentC'd} shows that $S^{(\hat{\phi})}(A'^{(M)})$ has order $2$. So the corollary is clear from the exact sequence \eqref{exactA}. Note that, of course, $k_-(M)$ has to be even since $M \equiv 1 \mod 4$.
\end{proof}

\smallskip

\begin{cor}\label{DescentCor3}
Assume that $M = -p_0 R_- \equiv 1 \mod 4$, where $p_0$ is a prime congruent to $3$ modulo $4$ and splitting in $\BQ(\sqrt{p})$. Then $\mathfrak S^{(\phi)}(A^{(M)})$ has exact order $2$, and $\mathfrak S^{(2)}(A^{(M)})$ has exact order $2$.
\end{cor}
\begin{proof}
The first assertion is clear from Proposition \ref{DescentCd}. Proposition \ref{DescentC'd} shows that $S^{(\hat{\phi})}(A'^{(M)})$ has order $2$, whence the assertion follows from the exact sequence \eqref{exactA}, and the fact that $\mathfrak S^{(2)}(A^{(M)})$ must have odd $\mathbb{F}_2$-dimension. Note that, of course, $r_-(M)$ has to be even since $M \equiv 1 \mod 4$.
\end{proof}

\smallskip

\begin{cor}\label{DescentCor4}
Assume that $M = q_0 R_- \equiv 1 \mod 4$, where $q_0$ is a prime congruent to $1$ modulo $4$ and inert in $\BQ(\sqrt{p})$. Then $\mathfrak S'^{(\phi)}(A^{(M)})$ has exact order $2$, and $\mathfrak S^{(2)}(A'^{(M)})$ has exact order $2$.
\end{cor}
\begin{proof}
The first assertion is clear from Proposition \ref{DescentC'd}. Proposition \ref{DescentCd} shows that $S^{(\phi)}(A^{(M)})$ has order $2$, whence the assertion follows from the exact sequence \eqref{exactA'}, and the fact that $\mathfrak S^{(2)}(A'^{(M)})$ must have odd $\mathbb{F}_2$-dimension. Note that, of course, $r_-(M)$ has to be even since $M \equiv 1 \mod 4$.
\end{proof}

\medskip

We now give the Tamagawa factors for the curves $A^{(M)}$ and $A'^{(M)}$, with a brief indication of proofs. We assume once again that $M$ is an arbitrary square-free integer, and write $D_M$ for the discriminant of the field $\BQ(\sqrt{M})$. Note that both $A^{(M)}$ and $A'^{(M)}$ have bad additive reduction at all primes dividing $pD_M$. Write $c_p(A^{(M)})$ for the Tamagawa factor of $A^{(M)}$ at a finite prime $p$, and similarly for $A'^{(M)}$. If $q$ is an odd prime of bad additive reduction, we have
\begin{equation}\label{TamAA'}
ord_2(c_q(A^{(M)})) = ord_2(\#A(\BQ_q)[2]), ord_2(c_q(A'^{(M)})) = ord_2(\#A'(\BQ_q)[2])
\end{equation}
by Lemma \ref{Tam_c_q}. We then have the following propositions.

\begin{prop}\label{Tam_A} For all odd square-free integers $M$, we have (i) $A^{(M)}(\BR)$ has one connected component, (ii) $c_p(A^{(M)}) = 2$,
(iii) $c_q(A^{(M)}) = 2$ if $q \equiv 3 \mod 4$, and (iv) $c_q(A^{(M)}) = 4$ if $q \equiv 1 \mod 4$.
\end{prop}
\begin{proof} Assertion (i) follows immediately from the fact that $\BQ(A[2])=\BQ(i)$. The remaining assertions involving odd primes $q$ of bad reduction follow immediately from \eqref{TamAA'}, on noting that $A(\BQ_q)[2]$ is of order $2$ or $4$, according as $q$ does not or does split in $\BQ(i)$, respectively.
\end{proof}

\smallskip

\begin{prop}\label{Tam_A'} For all odd square-free integers $M$, we have (i) $A'^{(M)}(\BR)$ has two connected components, (ii) $c_p(A'^{(M)})=1$, (iii) if $q$ is an odd prime dividing $M$, which is inert in $\BQ(\sqrt{p})$, then $c_q(A'^{(M)})=2$, (iv) if $q$ is an odd prime dividing $M$, which splits in $\BQ(\sqrt{p})$, then $c_q(A'^{(M)})=4$.
\end{prop}
\begin{proof} Assertion (i) follows immediately from the fact that $\BQ(A'[2])=\BQ(\sqrt{p})$. The remaining assertions involving odd primes of bad reduction follow immediately from \eqref{TamAA'}, on noting that $A'(\BQ_q)[2]$ is of order $2$ or $4$, according as $q$ does not or does split in $\BQ(\sqrt{p})$, respectively.
\end{proof}

\medskip

\subsection{Behaviour of Hecke eigenvalues}

Recall that the $L$-function of an elliptic curve $E$ over $\BQ$ is defined as an infinite Euler product
$$
L(E,s)=\prod_{q \nmid C} (1 - a_q q^{-s} + q^{1-2s})^{-1} \prod_{q \mid C} (1 - a_q q^{-s})^{-1} =: \sum a_n n^{-s},
$$
where
$$
a_q
=
\left\{
\begin{array}{llll}
q+1-\#E(\BF_q) & \hbox{if $E$ has good reduction at $q$,} \\
1              & \hbox{if $E$ has split multiplicative reduction at $q$,} \\
-1             & \hbox{if $E$ has non-split multiplicative reduction at $q$,} \\
0              & \hbox{if $E$ has additive reduction at $q$.}
\end{array}
\right.
$$
Here we give a result of the behaviour of the coefficients $a_q$ of the $L$-function of elliptic curve $A$.

\medskip

\begin{thm}\label{Thm a_q}
Let $q$ be an odd prime distinct with the conductor $p$ of $A$. Then we have that
$$
a_p=1;
$$
$$
a_2
=
\left\{
\begin{array}{ll}
-1 & \hbox{if $p \equiv 1 \ \mod 16$,} \\
1  & \hbox{if $p \equiv 9 \ \mod 16$;}
\end{array}
\right.
$$
and
$$
a_q
\equiv
\left\{
\begin{array}{ll}
2 \ \mod 4 & \hbox{if $q \equiv 1 \mod 4$,} \\
2 \ \mod 4 & \hbox{if $q \equiv 3 \mod 4$ and $q$ is inert in $\BQ(\sqrt{p})$,} \\
0 \ \mod 4 & \hbox{if $q \equiv 3 \mod 4$ and $q$ splits in $\BQ(\sqrt{p})$.}
\end{array}
\right.
$$
\end{thm}

\begin{proof}
The assertion for $a_p$ is clear as $A$ has split multiplicative reduction at $p$.

For $a_2$, we can do a straightforward calculation on the minimal form of $A$ modulo $2$, whence we get
$$
y^2 + xy \equiv x^3 + \frac{u-1}{4} x^2 + 1 \ \mod 2.
$$
When $u \equiv 1 \mod 8$, the above equation becomes $y^2 + xy \equiv x^3 + 1 \mod 2$, we then get $\#A(\BF_2)=4$, i.e. $a_2=-1$. When $u \equiv 5 \mod 8$, the above equation becomes $y^2 + xy \equiv x^3 \mod 2$, we then get $\#A(\BF_2)=2$, i.e. $a_2=1$. The assertion then follows by noting $p=u^2+64$.

For $a_q$, first note that the $2$-division field $\BQ(A[2])=\BQ(i)$ and $\BQ(A'[2])=\BQ(\sqrt{p})$, and we have the same $L$-function of $A$ and $A'$. So we have that $A(\BF_q)[2] \cong \BZ/2\BZ \times \BZ/2\BZ$ when $q \equiv 1 \mod 4$, and $A'(\BF_q)[2] \cong \BZ/2\BZ \times \BZ/2\BZ$ when $q$ splits in $\BQ(\sqrt{p})$. Since $A(\BF_q)[2]$ and $A'(\BF_q)[2]$ are subgroups of $A(\BF_q)$ and $A'(\BF_q)$, respectively. We have that $4 \mid \#A(\BF_q)$ and $4 \mid \#A'(\BF_q)$. Then the assertions for $q \equiv 1 \mod4$ and $q \equiv 3 \mod 4$ splitting in $\BQ(\sqrt{p})$ follow by applying $a_q=q+1-\#A(\BF_q)$. While for $q \equiv 3 \mod 4$ inert in $\BQ(\sqrt{p})$, we have that $A(\BF_q)[2] \cong \BZ/2\BZ$. It is easy to compute that $\BQ(\sqrt{p})$ is a subfield of $\BQ(A[4]^*)$, where $A[4]^*$ means any one of the $4$-division points which is deduced from the non-trivial rational $2$-torsion point of $A(\BQ)$. But $q$ is inert in $\BQ(\sqrt{p})$, that means $A(\BF_q)[4] = A(\BF_q)[2] \cong \BZ/2\BZ$. Hence $2 \mid \#A(\BF_q)$, but $4 \nmid \#A(\BF_q)$. Then the assertion follows in this case. This completes our proof.
\end{proof}

\medskip

\subsection{2-part of Birch and Swinnerton-Dyer conjecture}

For the elliptic curve $A$, we have the following results.

\begin{thm}\label{L(A,1)}
We have
$$
ord_2(L^{(alg)}(A,1))=-1
$$
for any prime $p$, with $p=u^2+64$ for some integer $u \equiv 5 \mod 8$.
\end{thm}

For the Neumann--Setzer elliptic curves, Stein and Watkins \cite{Stein} considered the parity of the modular degree of the map
$$
\phi: X_0(p) \to A.
$$
They proved that $deg(\phi)$ is odd if and only if $u \equiv 5 \mod 8$. Next we shall prove Theorem \ref{L(A,1)}. Before proving the theorem we first prove the following lemma.

\begin{lem}\label{lem_mod_par}
For the modular parametrizations of $A$ when $u \equiv 5 \mod 8$, $\phi([0])$ is non-trivial, and is precisely the non-trivial torsion point of order $2$.
\end{lem}
\begin{proof}
By the result of Stein and Watkins \cite{Stein}, the modular degree of $\phi$ is odd if and only if $u \equiv 5 \mod 8$. Let $J_0(p)$ be the Jacobian of $X_0(p)$, and let $B$ be the kernel of the map $J_0(p) \to A$. Denote $H$ to be the cuspidal subgroup, which is known to be the torsion subgroup of $J_0(p)(\BQ)$, and it is generated by $[0]-[\infty]$ and cyclic of order $n$, where $n$ is the numerator of $\frac{p-1}{12}$. We also have $A$ is contained in $J_0(p)$ and $A(\BQ)[2]=H[2]$. We now assume that $\phi([0])$ is trivial. Then $H$ is contained in $B$. Thus the cardinality of the intersection of $B$ and $A$ is even. Then by \cite[Lemma 2.2]{Stein}, the modular degree of $A$ should be even, contradiction. Note that $A(\BQ)\cong \BZ/2\BZ$, thus $\phi([0])$ is non-trivial, and is precisely the non-trivial $2$-torsion point.
\end{proof}

\smallskip

We now give the proof of Theorem \ref{L(A,1)}.

\begin{proof}[Proof of Theorem \ref{L(A,1)}]
When $u \equiv 5 \mod 8$, we have $a_2=1$. Then by the modular symbol formula \eqref{ms1}, we have
$$
(1+2-a_2)L(A,1)=-\langle \{0,\frac{1}{2}\}, f \rangle.
$$
Note that the point $[0]$ is equivalent to $[\frac{1}{2}]$ under $\Gamma_0(p)$, so $\langle \{0,\frac{1}{2}\}, f \rangle$ is an integral period (real) and in the period lattice $\Lambda_f$, i.e. $2*\phi([0]) \equiv 0 \mod \Lambda_f$, as $\phi([0])=-I_f(0)=L(A,1)$. Then by Lemma \ref{lem_mod_par}, $\phi([0])$ is the non-trivial $2$-torsion point, so the denominator of $\phi([0])/\Omega_f^+$ must be $2$, where $\Omega_f^+$ is the least positive real period of $A$. Thus
$$
ord_2(L^{(alg)}(A,1))=-1.
$$
This completes the proof of Theorem \ref{L(A,1)}.
\end{proof}

\medskip

A key point in the above proof is that the image of the cusp $[0]$ is the non-trivial $2$-torsion point under the parametrization when $u \equiv 5 \mod 8$. After investigating many numerical examples we make the following conjecture.
\begin{conj}\label{conj_L(A,1)}
For any prime $p$, with $p=u^2+64$ for some integer $u \equiv 1 \mod 4$, we have
$$
ord_2(L^{(alg)}(A,1))=-1.
$$
\end{conj}

\medskip

We now give the proof of Theorem \ref{thmA}.

\begin{proof}[Proof of Theorem \ref{thmA}]
Firstly, $A$ is the $\Gamma_0(p)$-optimal elliptic curve. By Theorem \ref{Thm a_q}, we have $N_q \equiv 2 \mod 4$ when $q \equiv 3 \mod 4$ and is inert in $\BQ(\sqrt{p})$. Secondly, note that when $u \equiv 5 \mod 8$, we have $ord_2(L^{(alg)}(A,1))=-1$ by Theorem \ref{L(A,1)}. These satisfy the assumptions in Theorem \ref{thm2}, we then have $ord_2(L^{(alg)}(A^{(-q)},1))=0$. Naturally, $L(A^{(-q)},s)$ does not vanish at $s=1$. Both $A^{(-q)}(\BQ)$ and $\Sha(A^{(-q)}(\BQ))$ are finite by the theorem of Kolyvagin. The cardinality of $\Sha(A^{(-q)}(\BQ))$ is odd as $\Sha(A^{(-q)}(\BQ))[2]$ is trivial by Corollary \ref{DescentCor1}. Then combining the results of Proposition \ref{Tam_A}, we know that the $2$-part of Birch and Swinnerton-Dyer conjecture holds for $A^{(-q)}$.
\end{proof}

\smallskip

We make the following proposition, which tells that the above result could be probably generalised to the twists of many prime factors. We denote $S''_m:=\sum_{k=1}^{m} \chi(k) \langle \{0,\frac{k}{m}\}, f \rangle$.

\begin{prop}
Let $M=q_1q_2 \cdots q_{2r}$, where $r$ is any positive integer and $q_1, q_2, \ldots, q_{2r}$ are distinct primes congruent to $3$ modulo $4$ and inert in $\BQ(\sqrt{p})$. Assuming Conjecture \ref{conj_L(A,1)} and $ord_2(S''_M/\Omega_f^+)=ord_2(S'_M/\Omega_f^+)$, we then have that
$$
ord_2(L^{(alg)}(A^{(M)},1))=2r-1,
$$
for any integer $u \equiv 1 \mod 4$. Hence $L(A^{(M)},s)$ does not vanish at $s=1$, and so $A^{(M)}(\BQ)$ is finite, the Tate--Shafarevich group $\Sha(A^{(M)}(\BQ))$ is finite of odd cardinality, and the $2$-part of Birch and Swinnerton-Dyer conjecture is valid for $A^{(M)}$.
\end{prop}

The above proposition follows easily by Lemma \ref{lemmaS'_m-2}, and combining the results given by Corollary \ref{DescentCor1} and Proposition \ref{Tam_A}.

Here we give some numerical examples supporting the above proposition. We take $p=73$, i.e. $u=-3$, whence the elliptic curve $A$ has a minimal Weierstrass equation given by
$$
y^2 + xy = x^3 - x^2 + 4x - 3.
$$
Moreover, $L^{(alg)}(A,1)=\frac{1}{2}$ and it has discriminant $-73^2$. Here are a list of primes which are congruent to $3$ modulo $4$ and inert in $\BQ(\sqrt{73})$:
$$
7, 11, 31, 43, 47, 59, 83, 103, 107, 131, 139, 151, 163, 167, 179, 191, 199, \ldots.
$$
We take $M$ to be the product of any two primes in the above list. We always have
$$
ord_2(S''_M/\Omega_f^+)=ord_2(S'_M/\Omega_f^+)=1.
$$
It then follows that
$$
ord_2(L^{(alg)}(A^{(M)},1))=1.
$$
It is clear that the $2$-part of Birch and Swinnerton-Dyer conjecture is valid for all of these twists of $A$.

\bigskip

\noindent Shuai Zhai, \\
School of Mathematics, Shandong University, \\
Jinan, Shandong 250100, China, and \\
Department of Pure Mathematics and Mathematical Statistics, \\
University of Cambridge, Cambridge CB3 0WB, UK. \\
{\it shuaizhai@gmail.com}


\end{document}